\documentclass[11pt,a4paper]{article}
\usepackage[margin=1in]{geometry}
\usepackage{amsmath,amssymb,amsthm,mathtools}
\usepackage{booktabs,tabularx,array}
\usepackage{microtype}
\usepackage{xcolor}
\usepackage{tikz}
\usepackage[hidelinks]{hyperref}
\numberwithin{equation}{section}
\usepackage[round]{natbib}

\newtheorem{theorem}{Theorem}[section]
\newtheorem{lemma}[theorem]{Lemma}
\newtheorem{proposition}[theorem]{Proposition}

\newtheorem{conjecture}[theorem]{Conjecture}
\theoremstyle{definition}
\newtheorem{definition}[theorem]{Definition}

\theoremstyle{remark}

\newcommand{\Rev}{\mathsf{Rev}}

\newcommand{\cB}{\mathcal{B}}

\newcommand{\PALM}{\mathrm{PALM}}
\newcommand{\dd}{\,\mathrm{d}}

\usepackage{setspace}
\definecolor{deepblue}{RGB}{31,78,121}
\definecolor{deepgreen}{RGB}{35,105,82}
\hypersetup{%
 colorlinks=true,
 linkcolor=deepblue,
 urlcolor=deepblue,
 citecolor=deepgreen,
}

\title{Assortment optimization under the perception-adjusted Luce model}
\author{Gerardo Berbeglia\thanks{Melbourne Business School,
                                 The University of Melbourne, Australia.
                                 E-mail:~\href{mailto:g.berbeglia@mbs.edu}{g.berbeglia@mbs.edu}}
        \and
        Mashbat Suzuki\thanks{School of Computer Science \& Engineering,
                              UNSW Sydney, Australia.
                              E-mail:~\href{mailto:mashbats@gmail.com}{mashbats@gmail.com}}
        \and 
        Matthew K.\ Tam\thanks{School of Mathematics \& Statistics,
                               The University of Melbourne, Australia.
                               E-mail:~\href{mailto:matthew.tam@unimelb.edu.au}{matthew.tam@unimelb.edu.au}}
        }
\date{\today}

\begin{document}
\maketitle

\begin{abstract}
The Perception-Adjusted Luce Model (PALM) introduced by \cite{echenique2018perception} is a sequential choice model that can accommodate violations of regularity and choice
overload, behaviors not captured by the utility-maximization framework. \cite{flores2019assortment} studied the assortment problem under the PALM and proved that, for the case of two levels, there is always an optimal solution that is revenue-ordered by level. They also conjectured that revenue-ordered by levels remain optimal for the general PALM with any number of levels. We resolve this question in the affirmative by showing that every PALM instance admits an optimal assortment that includes, within each level, all products at or above a level-specific revenue threshold. The proof involves first showing how the conjecture which deals with a discrete optimization problem, can be reduced to proving certain property of a related continuous optimization problem. Specifically, we study a continuous refinement of the assortment decision under the same PALM choice rule. Although the refinement permits partial utility selections, we prove it always has an optimum that selects complete revenue tiers. This allows us to show that the assortment problem is polynomial-time solvable when the number of levels is fixed.
\\ 

\end{abstract}

\section{Introduction}
In the field of revenue management, the assortment problem consists of choosing what subset of products to show to customers in order to maximize the expected revenue. This problem has multiple applications in a diverse set of industries including airlines, retail, accommodation, and hospitality. The difficulty of solving a given assortment problem is typically dependent on the way customer behavior is modeled. For some consumer choice models such as the multinomial logit model (MNL), the random consideration set choice model or the Markov chain choice model, the corresponding assortment problem can be solved efficiently \citep{talluri2004revenue, gallego2024random, blanchet2013markov}. For other choice models such as the latent class multinomial logit (LC-MNL) model, the exponomial choice model or the ranked-list choice model, the assortment problem is NP-hard \citep{bront2009column,lu2026nphardness, aouad2015approximability}.

Most assortment studies use a random utility framework in which consumers maximize
utilities drawn from a menu-independent distribution. These models, known as \textit{random utility models} (RUMs), satisfy a property called \textit{regularity} \citep{block1960random}: the probability of choosing an existing product cannot increase when another product is added to the offer set. However, substantial empirical evidence indicates that consumers often violate regularity and, consequently, that their behavior cannot always be explained by a RUM. One prominent example is the \textit{decoy effect} or \textit{attraction effect}, whereby adding an inferior alternative that is asymmetrically dominated by one of the existing products increases the probability that consumers choose the dominating product \citep{huber1982adding}. A related phenomenon is the \textit{compromise effect}: when a new product with extreme attribute values is introduced, an existing product may become the intermediate or “compromise” option and therefore more attractive \citep{simonson1989choice}. In both cases, the addition of an option can increase the choice probability of an existing product, violating regularity. A further departure from standard rational-choice predictions is the \textit{paradox of choice} or \textit{choice overload}, where the increase of the number of available alternatives can increase the likelihood that consumers postpone the decision or make no purchase at all \citep{iyengar2000choice,long2025choice}.

In the last 15 years, there has been an increased interest in studying consumer choice models that go beyond the RUM framework as a way to potentially explain and predict consumer choice behaviors such as the attraction effect and choice overload. One such model that has attracted considerable attention in the economics literature is the \textit{perception adjusted Luce model} (PALM) proposed by \cite{echenique2018perception}. Under the PALM, consumers make choices via a finite sequential process following a perception priority order which permits violations of regularity and choice overload. \cite{flores2019assortment} considered the special case of the PALM where there are only two stages (or priorities) and proved that the assortment problem is polynomial time solvable. Specifically, the authors proved that there is always an optimal solution to the assortment problem that is revenue-ordered at each of the two stages (i.e. in each level if a product with revenue $r$ is offered, every product with revenue of at least $r$ is also offered). They also conjectured that the revenue-ordered by level structure is optimal for the general PALM (with any number of stages).

In this paper we resolve this conjecture: for every finite number of perception levels, assortment optimisation under PALM admits an optimal assortment
that is revenue-ordered by level.  Our proof considers a continuous refinement
in which the firm may select part of each product's utility while keeping
revenues and perception levels fixed.  Conditional on the total utility
selected from each level, the best allocation selects utility in nonincreasing
revenue order.  The central step is to prove that the level totals can be
jointly optimised at revenue-tier boundaries.  At these boundaries, the refined
decision selects complete tiers and therefore corresponds to an ordinary
assortment.  Enumerating the resulting level-specific thresholds gives a
polynomial-time algorithm when the number of perception levels is fixed.

\subsection{Related literature}
 To position our paper in the broader literature, we review some of the most related choice models and assortment optimization results. We begin with a review of results under choice models that, like PALM, are based on the MNL model but unlike the PALM, fall within the class of choice models based on random utility. We then proceed to review the main choice models and assortment optimization results for models, like the PALM, fall outside the RUM class.

\paragraph{Related results: RUM based choice models}
The assortment optimization problem has already been studied for multiple choice models that are related to the PALM. One of the simplest heuristics to solve the assortment problem is to restrict the possible assortments to those that can be constructed as follows: choose a revenue threshold $r$ and then select all the products whose revenue is at least $r$. The revenue-ordered assortment heuristic consists of evaluating the expected revenue of only those $O(n)$ assortments (where $n$ denotes the number of products) and return the assortment that yields the highest expected revenue. Although this heuristic is not guaranteed to produce an optimal solution, it has been shown that it is optimal for one of the most studied choice model known as the Luce or multinomial choice model (MNL) \citep{luce1959,talluri2006theory}.

There are many choice models that extend the MNL as the PALM. One extension considers that consumers are divided into $k\geq 2 $ classes and each class behaves as a different MNL model. Under this model, known as the latent class MNL, revenue ordered assortments do not always yield an optimal solution to the assortment optimization problem. Moreover, it is shown that the problem is NP-hard \citep{bront2009column,rusmevichientong2014assortment}. \cite{blanchet2013markov} proposed a Markov chain choice model which extends the MNL choice model and showed that the assortment optimization can be solved efficiently. However, \cite{desir2020constrained} proved the problem becomes NP-hard when the assortment is required to satisfy some maximum cardinality requirement.
There has been a recent interest in the study of the assortment optimization problem for MNL based sequential choice models but unlike the PALM, they belong to the RUM class. \cite{liu2020assortment}  studied the assortment problem for a sequential choice model that extends the MNL called the multinomial logit model under sequential offerings. Unlike the PALM in which the perception levels are fixed, their model allows the firm to optimize the stage at which consumers are allowed to observe each product. Another difference with the PALM is that their model cannot explain consumer behaviors that falls outside the RUM class such as choice overload or the attraction effect. The authors showed that the assortment problem is NP-hard. Later, \cite{feldman2022multinomial} showed the existence of a PTAS for the assortment problem under the same model. \cite{gao2021assortment} considers a similar RUM model as in \cite{liu2020assortment} in which the utility of the outside option; rather than being re-sampled at each stage, it is sampled once at the beginning of the process. The authors proved that the assortment problem can be solved in polynomial time. \cite{wang2026assortment} and \cite{xu2023assortment} extend these sequential choice models by allowing customers to buy multiple products. In a related dynamic setting, \cite{tang2023performance} study multi-period assortment optimization under an MNL model in which product attractiveness depends on cumulative past sales.

\paragraph{Related results: choice models outside the RUM framework}
\cite{tversky1974judgment} and \cite{tversky1981framing} provide some of the earliest empirical evidence that individuals sometimes violate regularity and therefore choice models based on the random utility maximization (RUM) framework may fail to explain choice behavior. \cite{wernerfelt1995rational} shows how consumers can use the composition of market offerings to
infer product utilities, providing a rational explanation of the compromise effect. \cite{kamenica2008contextual} considers a model in which the consumers are uninformed and can infer the value they get from an alternative depending on the complete offer set and showed that the model can explain RUM violations. The model we consider in this paper was introduced by \cite{echenique2018perception} under the name of Perception Adjusted Luce Model (PALM). Under the PALM, each product has a predetermined perception level in the integer range $[k]=\{1,2,\hdots,k\}$. Consumer selection process occurs via a sequential mechanism and at each stage $i\in[k]$, the consumer probability to choose an item of perception level $i$ follows the multinomial logit formula taking into account all the products that are available in all levels. The authors provide a clean characterization of what choice models can be explained by the PALM and show that the model can explain violations to regularity and choice overload.  \cite{flores2019assortment} considers the sequential multinomial choice model (SMNL) in which products are divided into two categories and consumers follow the same sequential processes defined under the PALM. Under the SMNL, which is generalization of the MNL, \cite{flores2019assortment} proved that there is always an optimal solution that is revenue-ordered by level. The authors left the corresponding arbitrary-level statement for the PALM as an open conjecture.

\cite{echenique2019general} and \cite{ahumada2018luce} have independently proposed a model known as the General Luce model (GLM) in which consumers discard some products using consideration sets as a filter prior to selecting an alternative based on the MNL formula. The GLM allows violations to regularity. 
\cite{kleinberg2017comparison} proposed a choice model where each consumer has a ranking of all products and chooses the $\ell^{th}$ alternative from that ranking among those being offered. The authors prove that this model can incorporate violations of regularity as long as $k>1$. \cite{berbeglia2018generalized} extends this model by considering a consumer is sampled from a probability distribution among all possible rankings and integer values for $\ell$. The authors show that the resulting model can have better predictions than RUM models in some real datasets.

\subsection{Paper outline and proof strategy}
To establish the conjecture posed by \cite{flores2019assortment}, our proof strategy is as follows. In Section~\ref{sec: PALM+Refinement}, we first consider a refined version of the assortment problem in which the utility of each product can be chosen continuously, rather than requiring products to be either fully included or excluded. Without loss of optimality, we show that products within each level can be selected in decreasing order of revenue. This observation allows us to reformulate the refined problem in terms of the total utility allocated to each level, substantially reducing its complexity. As a result, we show that the conjecture would hold if the refined problem has an optimal solution whose allocated total utility at each level $i$ matches the total utility of the $\alpha_i$ products with highest revenue in level $i$ for some $\alpha_i \geq 0$. In Section~\ref{sec:model}, we then exploit the structure of the resulting objective function of the refined problem. One of the key insights in the proof is that, for the refined problem, every positive-valued local maximizer with at least two non-zero coordinates is necessarily utility-integral.  Finally, we provide some concluding remarks and some future research directions in Section~\ref{sec:conclusion}.

\section{Assortment Optimization under the Perception-Adjusted Luce Model}\label{sec: PALM+Refinement}

In this section, we formally introduce the Perception-Adjusted Luce Model (PALM), and formulate the traditional and refined assortment optimization under it. By analyzing the structure of the latter, we show that if an optimum of the refined assortment problem satisfies a property we call utility integral then there is an optimal solution to the traditional assortment optimization problem that is revenue-ordered-by-level.

\subsection{Traditional assortment optimization under the PALM }

The perception-adjusted Luce model (PALM), introduced by
\cite{echenique2018perception}, considers a finite product set $X$ partitioned
into ordered perception levels $X_1,\ldots,X_k$. Products in the same
level have equal priority, and lower-indexed levels are considered first.
This is equivalent to specifying a weak perception order. Each product
$g$ has utility $u_g>0$ and revenue $r_g\ge0$. The outside option represents
no purchase and has utility $z\ge0$ and revenue zero. 

For an assortment $S\subseteq X$, let $x=(x_1,\ldots,x_k)$ denote its associated vector, where
\begin{equation}\label{eq:x}
 x_i=\sum_{g\in S\cap X_i}u_g,
 \end{equation}
 and let \begin{equation}\label{eq:D}
 D(z,x)=z+\sum_{i=1}^k x_i.
 \end{equation}
 
Thus $x_i$ is the sum of the utility of the products from level $i$ that belong to $S$, and $D(z,x)$ consists of the sum of the utility of the products in $S$ plus the utility of the outside option.

\begin{definition}[PALM]\label{mod:PALM}
Given an assortment $S\subseteq X$, a consumer visits the levels in order. Conditional on reaching
level $i$, the consumer selects each product $g\in S\cap X_i$ with
probability $u_g/D(z,x)$, and move to the next level without purchasing with
probability $1-x_i/D(z,x)$. The unconditional probability of selecting product $g \in X_i$ is
\begin{equation}\label{rho_def}
 \mathcal{P}(g,S)= \frac{u_g}{D(z,x)}\prod_{h<i}\left(1-\frac{x_h}{D(z,x)}\right),
 \qquad g\in S\cap X_i.
\end{equation}
The probability of no purchase is $\prod_{i=1}^k(1-x_i/D(z,x))$.
For an empty assortment, no purchase occurs with probability one,
including when $z=0$.
\end{definition}

The denominator in PALM includes the entire assortment and remains the same at
every stage. With one nonempty level, PALM reduces to the Luce model;
with several nonempty levels, no purchase can have positive probability
even when $z=0$. For $D(z,x)>0$, the expected revenue obtained by offering the assortment $S$ whose associated vector is $x$ becomes
\begin{equation}\label{eq:ordinary-palm-revenue}
 \Rev^{\PALM}(S)
 =\sum_{i=1}^k
  \frac{\sum_{g\in S\cap X_i}u_gr_g}{D(z,x)}
  \prod_{h<i}\left(1-\frac{x_h}{D(z,x)}\right).
\end{equation}
The \textit{traditional assortment optimization problem} (or assortment problem in short) consists of finding an assortment $S \subseteq X$ that maximizes $\Rev^{\PALM}(S)$.

\begin{definition}[Revenue-ordered by level]\label{def:revenue-ordered-by-level}
An assortment $S\subseteq X$ is \emph{revenue-ordered by level} if, for
each level $i$, the inclusion of a product $g\in S\cap X_i$ in $S$ implies the inclusion of every product $h\in X_i$ with $r_h\ge r_g$. Equivalently, for each level $i$,
there is a threshold $\tau_i\in\{r_g:g\in X_i\}\cup\{+\infty\}$ such that
\[
 S\cap X_i=\{g\in X_i:r_g\ge\tau_i\}.
\]
\end{definition}

\citet{flores2019assortment} proved the following result for the case of a PALM with two levels and left open the conjecture that follows.

\begin{theorem}[Two-level SML \citep{flores2019assortment}]\label{thm:flores-two-level}
If $k=2$, the assortment problem under PALM has an optimal solution that is
revenue-ordered by level. 
\end{theorem}
\begin{conjecture}[Revenue-ordered by level for PALM \citep{flores2019assortment}]\label{conj:palm-revenue-order}The assortment problem under PALM admits an optimal solution that is revenue-ordered by level.
\end{conjecture}

A major obstacle to proving this conjecture is the interaction between assortment decisions across levels. Although the choice process is sequential, a product’s purchase probability depends on the assortments offered at its own level and at all preceding and following levels. Establishing the optimality of revenue ordering at every level therefore requires a careful analysis of these non-linear interactions and their effect on total revenue. We prove Conjecture~\ref{conj:palm-revenue-order} by studying a refined version of the assortment problem under the PALM which we introduce next.

\subsection{Refined assortment optimization under the PALM}
In the traditional assortment problem, each product is either offered in full or omitted. In the \emph{refined assortment problem under PALM}, the firm may reduce the utility of any product  $g$ to $\tilde u_g\in[0,u_g]$. \citet{berbeglia2021refined} studied the refined assortment optimization problem and showed that for several choice models the optimal refined assortment revenue is strictly higher than the traditional assortment revenue. To study the refined problem under the PALM, the revenues and perception levels remain the same and choice probabilities follow the same PALM rule with the selected utilities $\tilde u_g$. Taking $\tilde u_g\in\{0,u_g\}$ for every product recovers every traditional assortment.

For a refined decision, let $x_i=\sum_{g\in X_i}\tilde u_g$ denote the total selected utility at level $i$ and write
$$ D(z,x)=z+\sum_i x_i. $$

A key observation about the refined model is that, for fixed level totals $x_1,\dots,x_k$, the expected revenue is maximized by allocating utility to products in decreasing order of revenue within each level.  Indeed, transferring equal utility
from a lower-revenue product to higher-revenue product within the same level increases the expected revenue from that level without changing the revenue obtained from the other levels.

This observation allows the optimal allocation within each level to be described using a single-function. For each level $i$, let $U_i$ denote the total available utility given by 
$$ U_i=\sum_{g\in X_i}u_g.$$ 
Arrange the products in $X_i$ in nonincreasing revenue order and partition $[0,U_i]$ into consecutive intervals, assigning each product $g$ an interval of length $u_g$. Define $P_i:[0,U_i]\to\mathbb{R}_+$ to equal $r_g$ on the interval assigned to $g$. Thus, $P_i$ is a nonincreasing step function. Merging adjacent intervals with equal revenue gives its revenue tiers. The construction is illustrated in Figure~\ref{fig:palm-revenue-curves}.

The revenue at level $i$ obtainable from
utility $x_i\in[0,U_i]$ is the area under this curve:
\begin{equation}\label{eq:Ai}
 A_i(x_i)=\int_0^{x_i}P_i(t)\dd t.
\end{equation}
In other words, $A_i(x_i)$ is obtained by allocating $x_i$ of utility to the highest-revenue product first. Furthermore, since $P_i$ is a weakly decreasing step function, $A_i$ is continuous, nondecreasing, piecewise linear and concave. 

Let $\cB_i$ denote the set of breakpoints of $A_i$, together with
the endpoints $0$ and $U_i$. In other words, these are points where 
$P_i$ drops to a strictly lower revenue. We call $x_i\in[0,U_i]$ \emph{utility-integral} if
$x_i\in\cB_i$.
Figure~\ref{fig:palm-revenue-curves} illustrates
the construction.

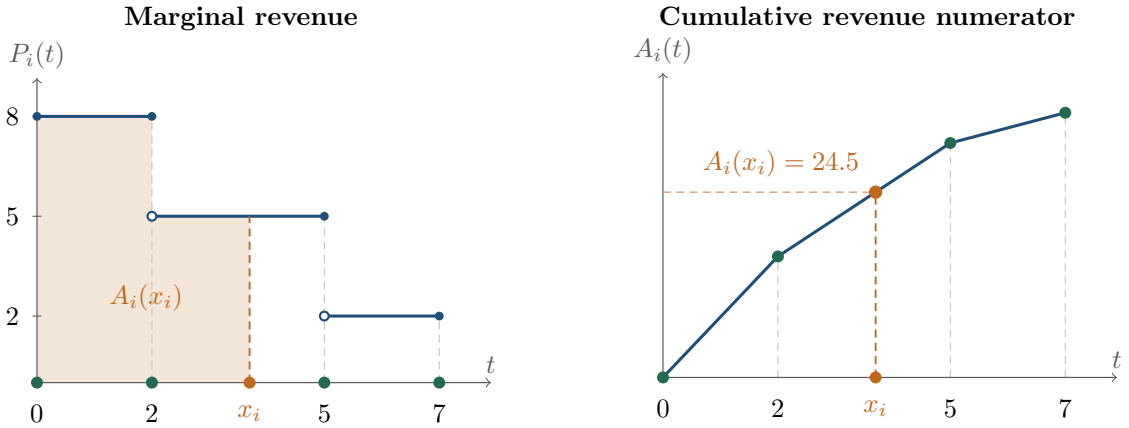
\begin{figure}[htbp]
\centering
\begingroup
\definecolor{tierorange}{RGB}{190,103,29}
\begin{minipage}[t]{0.48\textwidth}
\centering
\small\textbf{Marginal revenue}\par\vspace{4mm}
\begin{tikzpicture}[x=0.76cm,y=0.44cm,
  every node/.style={font=\small},
  axis/.style={->,black!65,line width=0.45pt},
  guide/.style={black!24,densely dashed,line width=0.4pt}]
  \path[use as bounding box] (-1.0,-1.65) rectangle (8.1,9.6);
  \fill[tierorange!17] (0,0)--(0,8)--(2,8)--(2,5)--(3.7,5)--(3.7,0)--cycle;
  \draw[guide] (2,0)--(2,8);
  \draw[guide] (5,0)--(5,5);
  \draw[guide] (7,0)--(7,2);
  \draw[axis] (0,0)--(7.9,0) node[above] {$t$};
  \draw[axis] (0,0)--(0,9.15) node[above] {$P_i(t)$};
  \foreach \h/\lab in {8/8,5/5,2/2}{
    \draw[black!65,line width=0.4pt] (-0.08,\h)--(0.08,\h);
    \node[anchor=east] at (-0.15,\h) {$\lab$};
  }
  \draw[deepblue,line width=1.15pt] (0,8)--(2,8);
  \draw[deepblue,line width=1.15pt] (2,5)--(5,5);
  \draw[deepblue,line width=1.15pt] (5,2)--(7,2);
  \foreach \x/\y in {0/8,2/8,5/5,7/2}{
    \fill[deepblue] (\x,\y) circle[radius=1.65pt];
  }
  \foreach \x/\y in {2/5,5/2}{
    \draw[deepblue,fill=white,line width=0.8pt]
      (\x,\y) circle[radius=1.65pt];
  }
  \draw[tierorange,densely dashed,line width=0.7pt] (3.7,0)--(3.7,5);
  \node[text=tierorange] at (1.9,2.55) {$A_i(x_i)$};
  \foreach \x in {0,2,5,7}{
    \fill[deepgreen] (\x,0) circle[radius=2.25pt];
  }
  \fill[tierorange] (3.7,0) circle[radius=2.25pt];
  \node[below=5pt] at (0,0) {$0$};
  \node[below=5pt] at (2,0) {$2$};
  \node[below=5pt,text=tierorange] at (3.7,0) {$x_i$};
  \node[below=5pt] at (5,0) {$5$};
  \node[below=5pt] at (7,0) {$7$};
\end{tikzpicture}
\end{minipage}\hfill
\begin{minipage}[t]{0.48\textwidth}
\centering
\small\textbf{Cumulative revenue numerator}\par\vspace{4mm}
\begin{tikzpicture}[x=0.76cm,y=0.10cm,
  every node/.style={font=\small},
  axis/.style={->,black!65,line width=0.45pt},
  guide/.style={black!24,densely dashed,line width=0.4pt}]
  \path[use as bounding box] (-1.0,-7.26) rectangle (8.1,42.24);
  \draw[guide] (2,0)--(2,16);
  \draw[guide] (5,0)--(5,31);
  \draw[guide] (7,0)--(7,35);
  \draw[axis] (0,0)--(7.9,0) node[above] {$t$};
  \draw[axis] (0,0)--(0,40.26) node[above] {$A_i(t)$};
  \draw[deepblue,line width=1.15pt] (0,0)--(2,16)--(5,31)--(7,35);
  \draw[tierorange,densely dashed,line width=0.7pt]
    (3.7,0)--(3.7,24.5);
  \draw[tierorange!65,densely dashed,line width=0.5pt]
    (0,24.5)--(3.7,24.5);
  \node[anchor=south east,text=tierorange] at (3.58,25.3) {$A_i(x_i)=24.5$};
  \foreach \x/\y in {0/0,2/16,5/31,7/35}{
    \fill[deepgreen] (\x,\y) circle[radius=2.25pt];
  }
  \fill[tierorange] (3.7,24.5) circle[radius=2.5pt];
  \fill[tierorange] (3.7,0) circle[radius=2.25pt];
  \node[below=5pt] at (0,0) {$0$};
  \node[below=5pt] at (2,0) {$2$};
  \node[below=5pt,text=tierorange] at (3.7,0) {$x_i$};
  \node[below=5pt] at (5,0) {$5$};
  \node[below=5pt] at (7,0) {$7$};
\end{tikzpicture}
\end{minipage}
\endgroup
\caption{A level with tier utilities $2,3,2$ and revenues $8,5,2$.
The shaded area under $P_i$ equals $A_i(x_i)$, shown on the right.
Here $x_i=3.7$ selects the entire first tier and part of the second,
giving $A_i(x_i)=24.5$. Green points mark the utility-integral choices
$\cB_i=\{0,2,5,7\}$, each of which selects only complete tiers.}
\label{fig:palm-revenue-curves}
\end{figure}

For a given $x$ with $D(z,x)>0$, the best refined revenue at these
level totals is
\begin{equation}\label{eq:refined}
 \Rev(z,x)
 =\sum_{i=1}^k\frac{A_i(x_i)}{D(z,x)}
  \prod_{h<i}\left(1-\frac{x_h}{D(z,x)}\right).
\end{equation}
This is the traditional assortment revenue formula
\eqref{eq:ordinary-palm-revenue}, except that each numerator is replaced by its
best fractional value $A_i(x_i)$. Every $x_i\in[0,U_i]$ is attainable
by a partial selection, so the refined assortment problem is
\begin{equation}\label{eq:refined-problem}
 \max_{0\le x_i\le U_i\ (i=1,\ldots,k)}\Rev(z,x).
\end{equation}
The all-zero selection has revenue zero, including when $z=0$.

Next, we relate utility integrality in the refined assortment problem to revenue ordering in the traditional assortment problem.
\subsection{From refined integrality to revenue-ordered assortments}

The following proposition turns a utility-integral optimum of the refined assortment problem (i.e. $x^*_i \in \cB_i$ for all $i$) into a revenue-ordered optimum of the traditional assortment problem.

\begin{proposition}[Existence of a revenue-ordered-by-level optimum]
\label{cor:PALM}
If the refined problem has an optimum with $x^*_i\in\cB_i$ for every
level, then the traditional assortment problem under the PALM has a revenue-ordered-by-level optimum with the
same revenue. 
\end{proposition}

\begin{proof}
For any ordinary assortment $S$, take $x_i=\sum_{g\in S\cap X_i}u_g$.
Since $A_i(x_i)$ is the largest numerator obtainable at this utility,
\[
 \sum_{g\in S\cap X_i}u_gr_g\le A_i(x_i).
\]
All coefficients multiplying these numerators in
\eqref{eq:ordinary-palm-revenue} are nonnegative. Hence
$\Rev^{\PALM}(S)\le\Rev(z,x)$, so the refined optimum is an upper
bound on every ordinary assortment's revenue.

Now let $x^*$ be a utility-integral refined optimum. In each level,
select complete tiers in decreasing revenue order until their cumulative
utility is $x_i^*$, and let $S^*$ be the resulting assortment. Its level
utilities are $x_i^*$ and its level numerators are exactly $A_i(x_i^*)$.
Consequently,
\begin{equation}\label{eq:refined PALM}
 \Rev^{\PALM}(S^*)=\Rev(z,x^*).
\end{equation}
Thus $S^*$ attains the refined upper bound and is optimal among all
ordinary assortments. The same conclusion holds for the empty selection
by the zero-revenue convention.
\end{proof}

In light of proposition \eqref{cor:PALM}, one way to prove the Conjecture \eqref{conj:palm-revenue-order} is to show that the refined assortment problem has an optimal solution that is utility-integral. We show this in the next section. 

\section{Utility-Integrality in the Refined Problem}\label{sec:model}
Our main theorem establishes utility integrality for the refined assortment problem under PALM.

\begin{theorem}[Utility-integrality of refined problem]\label{thm:main}
Let $k\ge1$ and $z\ge0$. Then the refined problem~\eqref{eq:refined-problem} has an optimal solution
$x^*$ that is utility-integral. In other words, $x^*=(x^*_1,\dots,x_k^*)$ with $x_i^*\in\cB_i$ for $i=1,\dots,k$.
\end{theorem}

Another consequence of Theorem~\ref{thm:main} is that under the PALM, the refined assortment optimization yields exactly the same revenue as the traditional assortment optimization. This extends a result of \citet{berbeglia2021refined} who showed that both problems yield the same revenue under the MNL choice model (a special case of the PALM with $k=1$ levels).

The remainder of this section is concerned with the proof of Theorem~\ref{thm:main}. The main difficulty in the analysis arises in the case when $z>0$ and $k\geq 2$. In Subsection~\ref{ss:2nd diff}, we establish a key intermediate result: if $z>0$ and $k\geq 2$, then every local maximizer of $\Rev(z,\cdot)$ with positive objective value and at least two non-zero coordinates is necessarily utility-integral. The proof of this result follows by combining first-order optimality conditions with a curvature argument that rules out utility-fractional coordinates. In Subsection~\ref{ss:main proof}, we use this result to establish Theorem~\ref{thm:main}, treating the limit cases where $z=0$ or at most one coordinate is non-zero separately.

\subsection{Utility-integrality of local optima}\label{ss:2nd diff}
In this section, we investigate utility-integral of local optima of the refined assortment optimization problem. In particular, we show that any local optimum with at least two non-zero coordinates is necessarily utility-integral. Recall that
 \[ 
  \Rev(z,x)
  =
  \sum_{i=1}^k
  \frac{A_i(x_i)}{D(z,x)}
  \prod_{r<i}\left(1-\frac{x_r}{D(z,x)}\right)\text{~~where~~}D(z,x)=z+\sum_{i=1}^k x_i.
 \]
Throughout, we use $\Rev_z(z,x)$ (resp.\ $\Rev_{x_i}(z,x)$) to represent the partial derivative of $\Rev$ with respect to $z$ (resp.\ $x_i$) evaluated at $(z,x)$.
\begin{lemma}
\label{lem:zderivative}
Let $z\geq 0$ and $x\geq 0$. For every feasible $x$ and every $0<\lambda\le1$, we have
\begin{equation}
\label{eq:scaling}
 \Rev(z,\lambda x)\ge \Rev(z/\lambda,x).
\end{equation}
Consequently, if $z>0$ and $x$ is a local maximizer of
$\Rev(z,\cdot)$, then
\begin{equation}
\label{eq:Rznegative}
 \Rev_z(z,x)\le0.
\end{equation}
\end{lemma}

\begin{proof}
If $z=0$ and $x=0$, then \eqref{eq:scaling} holds trivially, so suppose either $z>0$ or $x>0$. Then the survival factors on the two sides of \eqref{eq:scaling} agree, since
\[
 1-\frac{\lambda x_i}{z+\lambda\sum_r x_r}
 =
 1-\frac{x_i}{z/\lambda+\sum_r x_r}.
\]
Moreover, since $P_i$ is weakly decreasing, $t\mapsto A_i(t)/t$ is also weakly decreasing for $t>0$. Hence, for $x_i>0$, we have $\frac{A_i(\lambda x_i)}{\lambda x_i}
 \ge
 \frac{A_i(x_i)}{x_i}$. Consequently, for $x_i\geq 0$, we obtain
$$ A_i(\lambda x_i)\ge \lambda A_i(x_i). $$
Substituting this into the definition of $\Rev$ proves \eqref{eq:scaling}.

Now let $x$ be a local maximizer for $\Rev(z,\cdot)$ where $z>0$.
Given $\epsilon>0$, set $\lambda=z/(z+\epsilon).$  
By choosing $\epsilon$ sufficiently small, $\lambda x$ can be arbitrarily close to the local maximum $x$. Together with the identity $z/\lambda=z+\epsilon$, we therefore obtain
\[
 \Rev(z,x)
 \ge
 \Rev(z,\lambda x)
 \ge
 \Rev(z+\epsilon,x).
\]
Dividing by $\epsilon$ and letting $\epsilon\downarrow0$ gives $ \Rev_z(z,x)\le0$, and completes the proof.
\end{proof}

From herein, we suppose $z>0$ and $x>0$ unless otherwise stated. To simplify notation, we write $D$ as shorthand for $D(z,x)$ and introduce the following notation:
\begin{equation}
\label{eq:q-c-mu}
 q_i=\frac{x_i}{D},
 \quad
 q_0=\frac{z}{D}
 =1-\sum_{i=1}^kq_i>0,
 \quad
 c_i=\prod_{r<i}(1-q_r),
 \quad
 \mu_i=\frac{A_i(x_i)}{x_i}.
\end{equation}
Let $V_i$ denote the expected revenue conditional on reaching level $i$. Then, using the above notation, $V_i$ satisfies the recursion
\begin{equation}
\label{eq:continuation-values}
 V_{k+1}=0,
 \qquad
 V_i=q_i\mu_i+(1-q_i)V_{i+1}.
\end{equation}
Hence, in particular, we have $\Rev(z,x)=V_1$. Also denote the reach-weighted revenue margin at level $i$ by
\begin{equation}
\label{eq:margin-def}
 M_i:=c_i(\mu_i-V_{i+1}).
\end{equation}

\begin{lemma}
\label{lem:first-order-margins}
Let $z>0$, and let $x>0$ be a local maximizer of $\Rev(z,\cdot)$. Then
\begin{equation}
\label{eq:margin-sum}
 M_i\ge \sum_{\ell=1}^kq_\ell M_\ell=-D\Rev_z \qquad(i=1,\ldots,k).
\end{equation}
If, in addition, $\Rev(z,x)>0$, then $\Rev_z<0$ and hence  $M_i\ge-D\Rev_z>0$ for all $i=1,\ldots,k$.
\end{lemma}

\begin{proof}
We claim that 
\begin{equation}\label{eq:pV_i}
 D\frac{\partial V_i}{\partial z}=-\frac{1}{c_i}\sum_{\ell=i}^k q_\ell M_\ell\qquad (i=1,\dots,k+1).
\end{equation}
Since $V_{k+1}=0$, this holds trivially for $i=k+1$. Suppose \eqref{eq:pV_i} holds for $i+1$. Then differentiating \eqref{eq:continuation-values} w.r.t.\ $z$, followed by substituting \eqref{eq:pV_i} and the identity $c_{i+1}=(1-q_i)c_i$ yields
\begin{align*}
D\frac{\partial V_i}{\partial z} 
=  (1-q_i)D\frac{\partial V_{i+1}}{\partial z} - q_i(\mu_i-V_{i+1})
= -\frac{1-q_i}{c_{i+1}}\sum_{\ell=i+1}^k q_\ell M_\ell-\frac{q_i}{c_i}M_i 
= -\frac{1}{c_{i}}\sum_{\ell=i}^k q_\ell M_\ell.
\end{align*}
This establishes \eqref{eq:pV_i}.  In particular, since $\Rev(z,x)=V_1$ and $c_1=1$, the equality in \eqref{eq:margin-sum} follows.

Denote $w(t)=(z+t,x-te_i)$. Then taking the one sided derivative (\emph{i.e.,} $t\to0^+$) gives
$$ \Rev_z(z,x)-\Rev_{x_i}^-(z,x) = \frac{d_+}{dt}\bigl(\Rev(w(t))\bigr)\bigr|_{t=0} =  -\frac{c_i}{D}\bigl(A_i^-(x_i)-V_{i+1}\bigr), $$ 
where $\Rev^-_{x_i}$ and $A^-_i$ denote the one-sided derivatives w.r.t.\ $x_i$ from below. 
Since $A_i$ is concave and $A_i(0)=0$, we have $A^-_i(x_i)\leq \frac{A_i(x_i)-A_i(0)}{x_i-0}=\mu_i$ and so
$$ \Rev_{x_i}^-(z,x)-\Rev_z(z,x) =  \frac{c_i}{D}\bigl(A_i^-(x_i)-V_{i+1}\bigr) \leq \frac{c_i}{D}\bigl(\mu_i-V_{i+1}\bigr)=\frac{1}{D}M_i. $$ 
Combining this with the first equality in \eqref{eq:pV_i} yields
$$ \sum_{\ell=1}^kq_\ell M_\ell=-D\Rev_z\leq M_i-D\Rev_{x_i}^-(z,x). $$
Since $x$ is a local maximiser, we must have $\Rev_{x_i}^-(z,x)\geq 0$, which establishes the first claim.

Lemma~\ref{lem:zderivative} gives $\Rev_z\le0$. Suppose now that the
objective value is positive but $\Rev_z=0$. Then $M_i\ge0$ for $i=1,\ldots,k$.
while \eqref{eq:margin-sum} gives $\sum_iq_iM_i=0$. Since every $q_i>0$, we obtain $M_i=0$ for every $i$.
Now $M_k=c_k\mu_k$, because $V_{k+1}=0$, so $\mu_k=0$ and hence
$V_k=0$. Working backwards, if $V_{i+1}=0$, then
$0=M_i=c_i(\mu_i-V_{i+1})=c_i\mu_i$,
so $\mu_i=0$ and therefore $V_i=0$. Thus $V_1=0$, contradicting $\Rev(z,x)=V_1>0.$
Hence $\Rev_z<0$ as claimed, which completes the proof.
\end{proof}

\begin{lemma}
\label{lem:fractional-curvature}
Let $z>0$ and $x>0$. Suppose $x_j$ lies in the interior of a
revenue tier and $\Rev_{x_j}=0$. Then 
\begin{equation}
\label{eq:curvature-margin}
 \frac{D^2}{2}\Rev_{x_jx_j}
 =
 \sum_{i=1}^kq_i\alpha_iM_i,
\end{equation}
where $\alpha\in\mathbb{R}^k$ is given by
\begin{equation}
\label{eq:alpha-def}
 \alpha_i=
 \begin{cases}
 \displaystyle
 \sum_{r=i}^{j-1}\frac{q_r}{1-q_r},
 &i<j,\\[4mm]
 0,
 &i=j,\\[2mm]
 \displaystyle
 1-\sum_{r=j+1}^{i-1}\frac{q_r}{1-q_r},
 &i>j.
 \end{cases}
\end{equation}
Moreover, if $z>0$, $k\ge2$, $x$ is a local maximum of $\Rev(z,\cdot)$ and $D\Rev_z<0$, then $\Rev_{x_jx_j}>0$.
\end{lemma}

\begin{proof}
Throughout this proof, we use prime to denote derivatives with respect to $x_j$. We claim that
\begin{equation}
\label{eq:V-xj-general}
 DV_r'
 =
 \begin{cases}
 \displaystyle
 -\frac{1}{c_r}\sum_{\ell=r}^k q_\ell M_\ell,
 &r>j,\\[3mm]
 \displaystyle
 \frac{1}{c_r}\sum_{\ell=1}^{r-1} q_\ell M_\ell,
 &r\le j.
 \end{cases}
\end{equation} 
To see this, differentiating both sides of the identity $V_i=q_i\mu_i+(1-q_i)V_{i+1}$ yields
\begin{equation}\label{eq:1st diff full}
 V_i' = q_i'(\mu_i-V_{i+1})+(1-q_i)V_{i+1}'+q_i\mu_i'.
\end{equation}
In particular, since $Dq_i'=-q_i$ and $D\mu_i'=0$ for all $i\neq j$, this implies
\begin{equation}\label{eq:1st diff}
 V_i' = -\frac{q_i}{D}(\mu_i-V_{i+1})+(1-q_i)V_{i+1}' \quad\forall i\neq j.
\end{equation}
Then, by using definitions of $c_i$ and $M_i$ and rearranging, \eqref{eq:1st diff} becomes
\begin{equation}\label{eq:1st order recursion}
c_{i+1}DV_{i+1}'-c_{i}DV_i'=q_iM_i\quad\forall i\neq j. 
\end{equation}
For $r\leq j$, \eqref{eq:V-xj-general} follows from \eqref{eq:1st order recursion} by using $V_1'=\Rev_{x_j}=0$ and $c_1=1$. And, for $r>j$, \eqref{eq:V-xj-general} follows from \eqref{eq:1st order recursion} by using $V_{k+1}'=0$.

Next, since $D^2(-q_i/D)'=2q_i$ for all $i\neq j$, differentiating both sides of \eqref{eq:1st diff} followed by substituting \eqref{eq:1st diff} yields
\begin{equation}
\label{eq:V-xjxj-general-1}
 D^2V_i'' = 2D(V_{i+1}'-V_i')+(1-q_i)D^2V_{i+1}''\quad\forall i\neq j.
\end{equation}
For $i=j$, a similar calculation using $Dq_j'=1-q_j$, $D^2q_j''=-2(1-q_j)$ and $x_j\mu_j''=-2\mu_j'$  after differentiating both sides of \eqref{eq:1st diff full} gives
\begin{equation}
\label{eq:V-xjxj-general-2}
 D^2V_j'' = 2D(V_{j+1}'-V_j')-2DV_{j+1}'+(1-q_j)D^2V_{j+1}''.
\end{equation}
Combining \eqref{eq:V-xjxj-general-1} and \eqref{eq:V-xjxj-general-2} with $V_1'=V_{k+1}'=V_{k+1}''=0$  gives
\begin{equation}\label{eq:V1''}
\begin{aligned}
  \frac{D^2}{2}V_1'' 
  = \sum_{i=1}^kc_iD(V_{i+1}'-V_i') - c_jDV_{j+1}' + \frac{c_{k+1}}{2}D^2V_{k+1}''
  = \sum_{i=1}^{k-1}c_i q_i DV_{i+1}' - c_jDV_{j+1}'.
\end{aligned}
\end{equation}
Then, substituting the identity $c_i/c_{i+1}=1/(1-q_i)$ and \eqref{eq:V-xj-general} into \eqref{eq:V1''}, yields
\begin{align*}
     \frac{D^2}{2}\Rev_{x_jx_j} 
     &= \sum_{i=1}^{j-1}\frac{q_i}{1-q_i} \sum_{\ell=1}^{i}q_\ell M_\ell - \sum_{i=j+1}^{k-1}\frac{q_i}{1-q_i} \sum_{\ell=i+1}^kq_\ell M_\ell + \sum_{\ell=j+1}^kq_\ell M_\ell = \sum_{i=1}^kq_i\alpha_iM_i,
 \end{align*}
where $\alpha\in\mathbb{R}^k$ is given by \eqref{eq:alpha-def}. This establishes \eqref{eq:curvature-margin}.

Moreover, suppose  $z>0$, $k\ge2$, $x$ is a local maximum of $\Rev(z,\cdot)$ and $D\Rev_z<0$.  Then Lemma~\ref{lem:first-order-margins} gives $M_i\geq \sum_{i=1}^kq_iM_i=-D\Rev_z>0$ for $i=1,\dots,k$. Denote $\delta_i
 :=
 M_i-\sum_{\ell=1}^kq_\ell M_\ell
 \ge0$ and $\alpha_*:=\min_i\alpha_i$. Since $q_i\geq 0$ and $\sum_iq_i=1-q_0$, we have
\[
 \sum_iq_i\alpha_i\delta_i \geq \alpha_*\sum_iq_i\delta_i = 
 \alpha_*q_0\sum_{\ell=1}^kq_\ell M_\ell.
\]
Consequently, we obtain
\begin{align*}
 \frac{D^2}{2}\Rev_{x_jx_j}=\sum_iq_i\alpha_iM_i
 =
 \left(\sum_\ell q_\ell M_\ell\right)\left(\sum_iq_i\alpha_i\right)
 +\sum_iq_i\alpha_i\delta_i 
 \geq
 \left(\sum_\ell q_\ell M_\ell\right)
 \left(\sum_iq_i\alpha_i+q_0\alpha_*\right).
\end{align*}
Thus to establish positivity of $\Rev_{x_jx_j}$, it suffices to show $\sum_iq_i\alpha_i+q_0\alpha_*>0$.

To do so, we consider two cases based on the sign of $\alpha_*$. If $\alpha_*\geq 0$, then \eqref{eq:alpha-def} implies $\alpha_{j+1}=1$ if $j<k$ and $\alpha_{j-1}=q_{j-1}/(1-q_{j-1})$ if $j=k>1$. In both cases, $\alpha_r>0$ for some index $r$. Hence $\sum_{i}q_i\alpha_i>0$ and positivity follows. Else, if $\alpha_*<0$, then by \eqref{eq:alpha-def} we must have $\alpha_k=\alpha_*$. In this case, we use the definition of $\alpha$ to 
\begin{align*}
 \sum_iq_i\alpha_i+q_0\alpha_*
 &= \sum_{i<j}q_i\alpha_i+\sum_{i>j}q_i\alpha_i+q_0\alpha_k\\
 &= q_0+q_k
+
\sum_{r<j}
\frac{q_r}{1-q_r}\sum_{i\le r}q_i
+
\sum_{r=j+1}^{k-1}
\frac{q_r}{1-q_r}\sum_{i<r}q_i>0,
\end{align*}
where strict positivity follows because every term is nonnegative and $q_0=z/D>0$. This establishes the claim and completes the proof. 
\end{proof}

Given a feasible point $x=(x_1,\dots,x_k)$ for the refined assortment problem, we say level $i$ of $x$ is \emph{active} if $x_i>0$ and \emph{inactive} if $x_i=0$.
\begin{proposition}
\label{prop:no-unique-fractional}
Let $z>0$ and $x\geq 0$. Suppose $x$ is local maximizer of $\Rev(z,\cdot)$ with at least two active levels and $\Rev(z,x)>0$. Then $x$ is utility-integral.
\end{proposition}

\begin{proof}
Without loss of generality, we delete inactive levels and relabel the remaining levels without changing
their order. Note that this preserves local optimality. Since there are at least two active levels, we may therefore assume that $x>0$ and $k\geq 2$. 

Suppose that $x_j$ is a fractional coordinate. Since $x_j$ lies in the interior of its current revenue tier, the first-order optimality yields $\Rev_{x_j}=0$. By Lemma~\ref{lem:first-order-margins}, we have 
$\Rev_z<0$ and so Lemma~\ref{lem:fractional-curvature} implies $\Rev_{x_jx_j}>0$. 
This contradicts the second-order necessary condition $\Rev_{x_jx_j}\le0$ for a local maximum at an interior point. We therefore conclude that $x$ has no fractional coordinates.
\end{proof}

\subsection{Proof of Theorem~\ref{thm:main}}\label{ss:main proof}

We are finally ready to present the proof of our main theorem.

\begin{proof}[Proof of Theorem~\ref{thm:main}]
First suppose $z=0$. If the zero vector is an optimal solution, then there is nothing to prove, so suppose all optima are non-zero. Let $$p_{\max}:=\max\{P_i(t):t\in[0,U_i],\,i=1,\dots,k\}$$ be the largest value of any
revenue tier so that $A_i(x_i)\le p_{\max} x_i$ for all $i=1,\dots,k$. For $x\neq 0$, we have $D=\sum_{i=1}^k x_i>0$ and $q_i=\frac{x_i}{D}$, and so
 $\sum_ic_iq_i
 =
 1-\prod_i\left(1-\frac{x_i}{D}\right)
 \le1$. 
Consequently, we obtain the upper bound
\begin{equation}\label{eq:rev_ub}
  \Rev(0,x) = \sum_{i=1}^k\frac{A_i(x_i)}{D}c_i
 \leq
 p_{\max}\sum_{i=1}^k
 q_ic_i \leq p_{\max}.
\end{equation}
Let $i^*\in\{1,\dots,k\}$ be a level whose first tier has value $p_{\max}$, and let $b_{i^*}$ be the endpoint of the tier. That is, $P_{i^*}(t)=p_{\max}$ for $t\in [0,b_{i^*}]$. Then setting $x^*=b_{i^*}{e_{i^*}}$ gives 
$$\Rev(0,x^*) = \frac{A_{i^*}(b_{i^*})}{b_{i^*}} = p_{\max}. $$
In other words, $x^*$ is a feasible point that attains the upper bound \eqref{eq:rev_ub} with $x^*_i\in\cB_i$ for $i=1,\dots,k$. Thus $x^*$ is an optimal solution of the refined problem which is utility-integral.

Now suppose $z>0$. Then the function $\Rev(z,\cdot)$ is continuous on the set of nonnegative vectors. Since the feasible region of refined assortment problem~\eqref{eq:refined-problem} is the product of closed intervals, it is a compact set. Hence existence of a global optimum $x^*$ follows by the Extreme Value Theorem. If no coordinate of $x^*$ is fractional, there is nothing to prove, so suppose there is at least one. Without loss of generality, we can delete all inactive levels and relabel the remaining active levels without changing their order. Since an inactive level has $A_i(0)=0$ and survival factor one, this leaves the objective unchanged and preserves global optimality. Let $m$ denote the number of active levels.

If the optimal value is zero, then the all-zero vector is a utility-integral optimum. Hence we need only consider the case where the optimal value is positive. If $m\ge2$, then Proposition~\ref{prop:no-unique-fractional} contradicts the existence of a fractional coordinate of $x^*$, so it remains only to consider case with $m=1$. Inside any constant-revenue tier, we have $A_1(x)=a+px$ for some $a,p\in\mathbb{R}$. Thus $\Rev(z,x)=\frac{a+px}{z+x}$ and so 
\[
 \frac{\dd}{\dd x}\Rev(z,x)
 =
 \frac{pz-a}{(z+x)^2}.
\]
Since the sign of this derivative is constant throughout the tier, 
$\Rev(z,\cdot)$ is either monotone or constant on that tier. In either case, if an interior point of the tier is optimal, then one of its endpoints is also
optimal.

Thus in every case there exists a global optimum $x^*$ with $x_i^*\in\cB_i$ for all $i=1,\dots,k$. This completes the proof.
\end{proof}

\section{Concluding remarks and future research}
\label{sec:conclusion}
We studied assortment optimization, a central problem in revenue management,
under the perception-adjusted Luce model (PALM) of
\citet{echenique2018perception}. PALM is an early sequential extension of the
Luce model that accommodates violations of regularity and choice overload.
We proved that there exists an
optimal assortment that is revenue-ordered by level: within each level, all
products at or above a level-specific revenue threshold are offered. This
resolves the conjecture of \citet{flores2019assortment} and extends their
two-level result to the general PALM. Our result reduces the
assortment decision to the choice of one threshold per level and yields a
polynomial-time algorithm when the number of levels is fixed.

Methodologically, we established this discrete structural result using
continuous optimization techniques. Our technique is a substantial departure from the algebraic proof by  \cite{flores2019assortment} for $k=2$. We considered the refined assortment problem in which the firm may reduce the utilities of the products being offered. We then used utility-preserving exchanges and first- and second-order optimality conditions to show that the refined assortment problem has an optimum that corresponds to an
ordinary assortment, so the two problems have the same optimal revenue.

We provide two research directions worthy of further study. Operational restrictions such as cardinality and shelf-space constraints are common in many businesses. These constraints may prevent the
inclusion of all products above a revenue threshold and may invalidate the use of optimality conditions used in our proof. Characterizing the structure of optimal
constrained assortments, and developing algorithms with approximation
guarantees where exact optimization is difficult would be interesting.
Finally, the exactness of the continuous refinement raises a methodological
question beyond PALM. Our proof uses the common utility denominator and the
ordered sequence of perception levels. Identifying conditions under which
a similar argument applies to other sequential choice models, or to
restricted forms of customer heterogeneity, would provide a clearer picture about the scope of our
approach. When the refinement is no longer exact, quantifying the gap
between refined and traditional assortment revenues could help assess its
usefulness for constructing bounds and approximation algorithms.

\paragraph{Acknowledgments.}
MS is supported by the ARC Laureate Project FL200100204 on ``Trustworthy AI''.
The research of MKT was supported in part by Australian Research Council grant DP230101749.

\paragraph{Declaration of AI use.}
During the preparation of this work the authors used ChatGPT to explore potential proof strategies and assist with proofreading. The authors take full responsibility for the paper's content.

\bibliographystyle{plainnat}
\bibliography{references_discrete_choice.bib}

\end{document}